\documentclass[10pt, reqno]{amsart}
\usepackage{amsmath, amsthm, amscd, amsfonts, amssymb, graphicx, color}
\usepackage[bookmarksnumbered, colorlinks, plainpages]{hyperref}
\hypersetup{colorlinks=true,linkcolor=red, anchorcolor=green, citecolor=cyan, urlcolor=red, filecolor=magenta, pdftoolbar=true}
\usepackage{mathrsfs}

\newtheorem{theorem}{Theorem}[section]
\newtheorem{lemma}[theorem]{Lemma}

\newtheorem{corollary}[theorem]{Corollary}
\theoremstyle{definition}
\newtheorem{definition}[theorem]{Definition}
\newtheorem{example}[theorem]{Example}

\theoremstyle{remark}
\newtheorem{remark}[theorem]{Remark}
\numberwithin{equation}{section}

\begin{document}
\setcounter{page}{1}

\title[Minkowski function]{K-theory of Minkowski question-mark function}


\author[Nikolaev]
{Igor V. Nikolaev$^1$}

\address{$^{1}$ Department of Mathematics and Computer Science, St.~John's University, 8000 Utopia Parkway,  
New York,  NY 11439, United States.}
\email{\textcolor[rgb]{0.00,0.00,0.84}{igor.v.nikolaev@gmail.com}}





\begin{abstract}
It is proved that  the Minkowski question-mark function 
comes from  the $K$-theory of  Cuntz-Pimsner algebras. 
We apply this result to calculate the action of Frobenius 
endomorphism at the infinite prime. Such a problem was 
raised by Serre and Deninger in the theory of  
local factors  of  zeta functions of projective varieties. 
\end{abstract}

\maketitle

\section{Introduction}
In 1904 Minkowski introduced a remarkable  function $?(x): [0,1]\to [0,1]$ 
mapping continued fraction $[0; a_1, a_2,\dots]$ of the real number $x\in [0,1]$
to the binary code $\sum_{k=0}^{\infty} \frac{\{0;1\}}{2^k}$ of a real number $?(x)\in [0,1]$ 
according to the formula [Minkowski 1904] \cite{Min1}:
\begin{equation}\label{eq1.1}
?([0; a_1, a_2, a_3\dots])=0, \underbrace{0,\dots,0}_{a_1-1},
 \underbrace{1,\dots,1}_{a_2},
  \underbrace{0,\dots,0}_{a_3}, 
  \dots
\end{equation}
The Minkowski question-mark function  (\ref{eq1.1}) can be written 
in an equivalent form of the convergent series (\ref{eq2.1}). 
Among many of the outstanding properties of $?(x)$ are the
following: (i) rational numbers are mapped to the dyadic rationals;
(ii) quadratic irrational numbers to the non-dyadic rationals; 
and (iii) non-quadratic irrational numbers to the irrational numbers
 [Minkowski 1904] \cite{Min1}.

 There exists an extension  of $?(x)$ to the higher dimensions [Panti 2008] \cite{Pan1}.
Namely, an $m$-dimensional Minkowski question-mark function $?^m: \mathbf{R}^m/\mathbf{Z}^m\to  \mathbf{R}^m/\mathbf{Z}^m$
is a one-to-one continuous function which maps: (i) $m$-tuples of rational numbers to such of the dyadic rationals;
(ii) $m$-tuples of algebraic numbers of degree $m+1$ over $\mathbf{Q}$ to such of the non-dyadic rational numbers;
and (iii) $m$-tuples of remaining irrational numbers to such of the irrational numbers [Minkowski 1904] \cite[case $m=1$]{Min1}
and [Panti 2008] \cite[case $m\ge 1$]{Pan1}.

The aim of our note is a $K$-theory of the Minkowski question-mark function,
 see Theorem \ref{thm1.3}. 
Namely, it is proved that such a function comes  from the $K$-theory of the Cuntz-Pimsner
algebras  $\mathcal{O}_{A_{\infty}}$  [Pask \& Raeburn 1996] \cite{PasRae1}
and  $C^*$-algebras $\mathbb{A}_{\Theta}$  attached to the Riemann surfaces \cite{Nik2} .
This result is applied 
in  the theory of local factors of zeta functions of projective varieties
[Serre 1970] \cite{Ser1} and [Deninger 1991]  \cite{Den1}.
In particular, we  calculate the action of Frobenius endomorphism at  the infinite prime
in terms of the Cuntz-Pimsner algebra $\mathcal{O}_{A_{\infty}}$, 
see Corollary \ref{cor1.4}.  To formalize our results, we use 
the following notation.

Let  $S_{g,n}$ be a  Riemann surface   of genus $g\ge 0$  with  $n\ge 0$ cusps.
Consider a cluster $C^*$-algebra $\mathbb{A}(S_{g,n})$ attached to 
 triangulation of  $S_{g,n}$  [Williams 2014]  \cite[Section 3.3]{Wil1} and \cite{Nik2}. 
Let $\{I_{\Theta}\subset \mathbb{A}(S_{g,n}) ~|~\Theta\in\mathbf{R}^{6g-7+2n}\}$
be a two-sided primitive ideal of $\mathbb{A}(S_{g,n})$ \cite[Theorem 2]{Nik2}. 
Such an ideal gives rise to a pair of distinct but closely related $C^*$-algebras:
(i) the quotient $C^*$-algebra $\mathbb{A}_{\Theta}:=\mathbb{A}(S_{g,n}) / I_{\Theta}$
by  the closed two-sided ideal $I_{\Theta}\subset \mathbb{A}(S_{g,n})$
\cite[Theorem 2]{Nik2}; 
(ii) the Cuntz-Pimsner algebra $\{\mathcal{O}_{A_{\infty}} ~|~A_{\infty}\in GL_{\infty}(\mathbf{Z})\}$ given by the 
formula  $\mathcal{O}_{A_{\infty}}\otimes\mathcal{K}\cong I_{\Theta}\rtimes_{\hat\alpha}\mathbb{T}$,
where $\mathcal{K}$ is the
$C^*$-algebra of compact operators, $\mathbb{T}\cong \mathbb{R}/\mathbb{Z}$  
and the crossed product is taken by  the Takai dual $\hat\alpha$ of an automorphism $\alpha$
of   $I_{\Theta}$ [Pask \& Raeburn 1996] \cite[Section 2.2]{PasRae1}.  
\begin{remark}\label{rmk1.1}
An explicit construction of the matrix  $A_{\infty}\in GL_{\infty}(\mathbf{Z})$ 
from  the Jacobi-Perron continued fraction of  $\Theta\in\mathbf{R}^{6g-7+2n}$
 [Bernstein 1971]  \cite{B} is given by formulas
 (\ref{eq3.1})-(\ref{eq3.2}) and Lemma \ref{lm3.2},
 see also Example \ref{ex3.1}.  
\end{remark}
In view of Lemma \ref{lm3.5}, 
the $K$-theory of  the $C^*$-algebras   $\mathcal{O}_{A_{\infty}}$
 [Pask \& Raeburn 1996] \cite[Theorem 3]{PasRae1} and 
$\mathbb{A}_{\Theta}$ \cite[Section 1]{Nik2} 
gives rise to  an  injective map:
\begin{equation}\label{eq1.4}
i:  K_0(\mathcal{O}_{A_{\infty}})\to K_0(\mathbb{A}_{\Theta}). 
\end{equation}
\begin{remark}\label{rmk1.2}
 The map (\ref{eq1.4}) is not a group  isomorphism of 
$K_0(\mathcal{O}_{A_{\infty}})$ to a subgroup of  $K_0(\mathbb{A}_{\Theta})$. 
Indeed, the abelian group $K_0(\mathcal{O}_{A_{\infty}})$ is infinite torsion 
[Pask \& Raeburn 1996] \cite[Theorem 3]{PasRae1} while 
$K_0(\mathbb{A}_{\Theta})$ is a torsion-free group \cite{Nik2}. 
\end{remark}
\begin{remark}\label{rmk1.2+}
It is known that  $K_0(\mathbb{A}_{\Theta})\cong \mathbf{Z}^{6g-6+2n}$,
where the lattice $\Lambda_{\Theta}=\mathbf{Z}+\mathbf{Z}\theta_1+\dots+\mathbf{Z}\theta_{6g-7+2n}$
defines $\mathbb{A}_{\Theta}$ up to an isomorphism of the $C^*$-algebras \cite{Nik2}. 
Then  (\ref{eq1.4}) gives rise to a map $i_*:   \mathbf{R}^{6g-7+2n}/\mathbf{Z}^{6g-7+2n}\to
 \mathbf{R}^{6g-7+2n}/\mathbf{Z}^{6g-7+2n}$ by the formula $\Theta\mapsto  \Lambda_{\Theta} \mod \mathbf{Z}$. 
 Conversely, each map $i_*$ defines a map  (\ref{eq1.4}).  
\end{remark}
In view of  Remark \ref{rmk1.2+}, 
our main results can be formulated as follows.   
\begin{theorem}\label{thm1.3}
The map $i:  K_0(\mathcal{O}_{A_{\infty}})\to K_0(\mathbb{A}_{\Theta})$
coincides pointwise with the   Minkowski question-mark function. 
  \end{theorem}

An application of Theorem \ref{thm1.3} is as follows. 
Let $V(k)$ be an $n$-dimensional smooth projective variety
over a number field $k$.  
Recall that the action of Frobenius endomorphism $Fr_p^i$
on the $i$-th $\ell$-adic cohomology group $H^i(V)$ of $V$  can be
extended to a prime $p$ ``at infinity''  [Serre 1970] \cite{Ser1}. 
Namely, there exists a  Frobenius endomorphism $Fr_{\infty}^i\in GL_{\infty}(\mathbf{Z})$
acting on an infinite-dimensional cohomology groups $H^i_{ar}(V)$ of  $V$
[Deninger 1991]  \cite{Den1}. 
The characteristic polynomial of $Fr_{\infty}^i$ is known to 
satisfy $\Gamma^i_V(s)\equiv char^{-1} Fr_{\infty}^i$,
where  $\Gamma^i_V(s)$ is the local factor at infinity  [Deninger 1991]  \cite[Theorem 4.1]{Den1}. 

We  consider a quantum invariant $(\Lambda, [I], K)$
of $V(k)$  consisting of the number field $K\subset\mathbf{R}$, an order 
$\Lambda\subseteq O_K$ in the ring of integers of $K$ and an ideal class $[I]\subset\Lambda$
\cite[Theorem 1.3]{Nik1}.  
Let  $\{(\Lambda^i, [I], K) ~|~ 0\le i\le 2n\}$ be the subrings $\Lambda^i\subseteq\Lambda$ as specified 
in \cite[p. 271]{Nik3}.  
Since $(\Lambda^i, [I], K)\cong (K_0(\mathbb{A}_{\Theta_i}), K_0^+(\mathbb{A}_{\Theta_i}))$ \cite[Section 7.3]{Bla},
one gets  a set of matrices $\{A_{\infty}^i\in GL_{\infty}(\mathbf{Z}) ~|~0\le i\le 2n\}$, see Remark \ref{rmk1.1}.  
Finally, let $\cong$ be the similarity relation between matrices in the group $GL_{\infty}(\mathbf{Z})$.  
One gets the following explicit formulas for the action of  Frobenius  endomorphisms $Fr_{\infty}^i$ at the infinite prime. 
\begin{corollary}\label{cor1.4}
$Fr_{\infty}^i\cong A_{\infty}^i$, where $0\le i\le 2n$. 
\end{corollary}
The paper is organized as follows.  A brief review of the preliminary facts is 
given in Section 2. Theorem \ref{thm1.3} and Corollary \ref{cor1.4} 
are proved in Section 3.

\section{Preliminaries}
We briefly review Minkowski question-mark functions,  Cuntz-Pimsner algebras  and cluster  
$C^*$-algebras. 
We refer the reader to   [Panti 2008] \cite{Pan1},    [Pask \& Raeburn 1996] \cite{PasRae1}
and \cite{Nik2} for a detailed exposition.

\subsection{Minkowski question-mark function}
Minkowski question-mark function is defined by the convergent
series
 \begin{equation}\label{eq2.1} 
 ?(x):=a_0+2\sum_{k=1}^{\infty} \frac{(-1)^{k+1}}{2^{a_1+\dots+a_k}},
 \end{equation}
where $x=[a_0,a_1,a_2,\dots]$ is the  continued fraction of the irrational number $x$. 
The $?(x): [0,1]\to [0,1]$ is a monotone continuous function with the following properties
 [Minkowski 1904] \cite[p. 172]{Min1}:

\medskip
(i) $?(0)=0$ and $?(1)=1$; 

\smallskip
(ii)  $?(\mathbf{Q})=\mathbf{Z}[\frac{1}{2}]$ are dyadic rationals;

\smallskip
(iii) $?(\mathscr{Q})=\mathbf{Q}-\mathbf{Z}[\frac{1}{2}]$, where $\mathscr{Q}$ are quadratic 
irrational numbers.

\bigskip
An $m$-dimensional  generalization of properties (i)-(iii) is as follows. 
\begin{theorem}\label{thm2.1}
{\bf (\cite[Theorem 2.1]{Pan1})}
There exists a unique $m$-dimensional Minkowski question-mark function $?^m: \mathbf{R}^m/\mathbf{Z}^m\to  \mathbf{R}^m/\mathbf{Z}^m$
which is one-to-one,  continuous and maps: 

\medskip
(i) $m$-tuples of the rational numbers to such of the dyadic rationals;

\smallskip
(ii) $m$-tuples of the algebraic numbers of degree $m+1$ over $\mathbf{Q}$ to such of the non-dyadic rational numbers;

\smallskip
(iii) $m$-tuples of remaining irrational numbers to such of the irrational numbers.
\end{theorem}
 \begin{remark}\label{rmk2.2}
 Panti's Theorem \cite[Theorem 2.1]{Pan1} was stated in terms of a unique homeomorphism 
 $\Phi:\mathbf{R}^m/\mathbf{Z}^m\to  \mathbf{R}^m/\mathbf{Z}^m$
conjugating  the tent map $T$ and  the M\"onkemeyer map $M$  [Panti 2008] \cite[Section 2]{Pan1}. 
 The reader can verify that the maps $T$ and $M$ characterize a unique  $m$-dimensional Minkowski 
 question-mark function $?^m: \mathbf{R}^m/\mathbf{Z}^m\to  \mathbf{R}^m/\mathbf{Z}^m$ satisfying properties (i)-(iii)
 of Theorem \ref{thm2.1}. 
 \end{remark} 

\subsection{Cuntz-Pimsner algebras}
The Cuntz-Krieger algebra $\mathcal{O}_A$ is a  $C^*$-algebra
generated by the  partial isometries $s_1,\dots, s_n$ which satisfy  the relations
\begin{equation}\label{eq2.2}
\left\{
\begin{array}{ccc}
s_1^*s_1 &=& a_{11} s_1s_1^*+a_{12} s_2s_2^*+\dots+a_{1n}s_ns_n^*\\ 
s_2^*s_2 &=& a_{21} s_1s_1^*+a_{22} s_2s_2^*+\dots+a_{2n}s_ns_n^*\\ 
                  &\dots&\\
s_n^*s_n &=& a_{n1} s_1s_1^*+a_{n2} s_2s_2^*+\dots+a_{nn}s_ns_n^*,             
\end{array}
\right.
\end{equation}
where $A=(a_{ij})$ is a square matrix with  $a_{ij}\in \{0, 1\}$. 
The Cuntz-Pimsner algebra  $\mathcal{O}_{A_{\infty}}$ corresponds to the case of 
 the countably infinite matrices $A_{\infty}\in GL_{\infty}(\mathbf{Z})$  
[Pask \& Raeburn 1996] \cite{PasRae1}.
The matrix $A_{\infty}$ is called row-finite,  if for each $i\in\mathbf{N}$
the number of $j\in\mathbf{N}$ with $a_{ij}\ne 0$ is finite.  The matrix $A_{\infty}$ is 
said to be irreducible, if some power of $A_{\infty}$ is a strictly positive matrix and $A_{\infty}$ is not a
permutation matrix.  If  $A_{\infty}$ is row-finite and irreducible, then the 
Cuntz-Pimsner algebra  $\mathcal{O}_{A_{\infty}}$ is a well-defined  and simple 
[Pask \& Raeburn 1996] \cite[Theorem 1]{PasRae1}.

An AF-core $\mathscr{F}\subset \mathcal{O}_{A_{\infty}}$ is an Approximately Finite (AF-) $C^*$-algebra
defined by the closure of  the infinite union $\cup_{k,j} \cup_{i\in V_k^j} \mathscr{F}_k^j(i)$,
where  $\mathscr{F}_k^j(i)$ are finite-dimensional $C^*$-algebras 
built from matrix $A_{\infty}$  [Pask \& Raeburn 1996] \cite[Definition 2.2.1]{PasRae1}. 
Let $\alpha: \mathcal{O}_{A_{\infty}}\to \mathcal{O}_{A_{\infty}}$ be an automorphism 
acting on the generators $s_i$ of $\mathcal{O}_{A_{\infty}}$ by
to the formula $\alpha_z(s_i)=zs_i$, where $z$ is a complex number  $|z|=1$.
One gets an action of the abelian group $\mathbb{T}\cong\mathbf{R}/\mathbf{Z}$ on  $\mathcal{O}_{A_{\infty}}$. 
The Takai duality [Pask \& Raeburn 1996] \cite[p. 432]{PasRae1} says that:
\begin{equation}\label{eq2.3} 
\mathscr{F}\rtimes_{\hat\alpha}\mathbb{T}\cong \mathcal{O}_{A_{\infty}}\otimes\mathcal{K},
\end{equation}
where $\hat\alpha$ is the Takai dual of $\alpha$ and $\mathcal{K}$ is the
$C^*$-algebra of compact operators.  Using (\ref{eq2.4}) one can calculate the the 
$K$-theory of  $\mathcal{O}_{A_{\infty}}$. Namely, the following statement is true. 
\begin{theorem}\label{thm2.3}
{\bf (\cite[Theorem 3]{PasRae1})}
If $A_{\infty}$ is row-finite irreducible matrix, then there exists an exact 
sequence of the abelian groups:
\begin{equation}\label{eq2.4} 
0\to K_1(\mathcal{O}_{A_{\infty}})\to \mathbf{Z}^{\infty}\buildrel 1-A_{\infty}^t\over\longrightarrow 
 \mathbf{Z}^{\infty}\buildrel i_*\over\longrightarrow  K_0(\mathcal{O}_{A_{\infty}})\to 0, 
\end{equation}
so that $K_0(\mathcal{O}_{A_{\infty}})\cong  \mathbf{Z}^{\infty}/(1-A_{\infty}^t) \mathbf{Z}^{\infty}$ and 
 $K_1(\mathcal{O}_{A_{\infty}})\cong Ker ~(1-A_{\infty}^t)$, where  $A_{\infty}^t$ is the transpose 
 of  $A_{\infty}$ and $i: \mathscr{F}\hookrightarrow \mathcal{O}_{A_{\infty}}$.
 Moreover, the Grothendieck semigroup $K_0^+(\mathscr{F})\cong \varinjlim (\mathbf{Z}^{\infty}, A_{\infty}^t)$. 
 \end{theorem}

\subsection{Cluster  $C^*$-algebras}
\subsubsection{Laurent phenomenon}
The cluster algebra  of rank $n$ 
is a subring  $\mathcal{A}(\mathbf{x}, B)$  of the field  of  rational functions in $n$ variables
depending  on  variables  $\mathbf{x}=(x_1,\dots, x_n)$
and a skew-symmetric matrix  $B=(b_{ij})\in M_n(\mathbf{Z})$.
The pair  $(\mathbf{x}, B)$ is called a  seed.
A new cluster $\mathbf{x}'=(x_1,\dots,x_k',\dots,  x_n)$ and a new
skew-symmetric matrix $B'=(b_{ij}')$ is obtained from 
$(\mathbf{x}, B)$ by the   exchange relations [Williams 2014]  \cite[Definition 2.22]{Wil1}:
\begin{eqnarray}
x_kx_k'  &=& \prod_{i=1}^n  x_i^{\max(b_{ik}, 0)} + \prod_{i=1}^n  x_i^{\max(-b_{ik}, 0)},\cr \nonumber
b_{ij}' &=& 
\begin{cases}
-b_{ij}  & \mbox{if}   ~i=k  ~\mbox{or}  ~j=k\cr
b_{ij}+\frac{|b_{ik}|b_{kj}+b_{ik}|b_{kj}|} {2}  & \mbox{otherwise.}
\end{cases}
\end{eqnarray}
The seed $(\mathbf{x}', B')$ is said to be a  mutation of $(\mathbf{x}, B)$ in direction $k$.
where $1\le k\le n$.  The  algebra  $\mathcal{A}(\mathbf{x}, B)$ is  generated by the 
cluster  variables $\{x_i\}_{i=1}^{\infty}$
obtained from the initial seed $(\mathbf{x}, B)$ by the iteration of mutations  in all possible
directions $k$.   The  Laurent phenomenon
 says  that  $\mathcal{A}(\mathbf{x}, B)\subset \mathbf{Z}[\mathbf{x}^{\pm 1}]$,
where  $\mathbf{Z}[\mathbf{x}^{\pm 1}]$ is the ring of  the Laurent polynomials in  variables $\mathbf{x}=(x_1,\dots,x_n)$
 [Williams 2014]  \cite[Theorem 2.27]{Wil1}.
 The cluster algebra  $\mathcal{A}(\mathbf{x}, B)$  has the structure of an additive abelian
semigroup consisting of the Laurent polynomials with positive coefficients. 
In other words,  the $\mathcal{A}(\mathbf{x}, B)$ is a dimension group
[Blackadar  1986] \cite[Section 7.3]{Bla}. 
The cluster $C^*$-algebra  $\mathbb{A}(\mathbf{x}, B)$  is   an  AF-algebra,  
such that $K_0(\mathbb{A}(\mathbf{x}, B))\cong  \mathcal{A}(\mathbf{x}, B)$.

\subsubsection{Cluster $C^*$-algebra $\mathbb{A}(S_{g,n})$}
Denote by $S_{g,n}$  the Riemann surface   of genus $g\ge 0$  with  $n\ge 0$ cusps.
 Let   $\mathcal{A}(\mathbf{x},  S_{g,n})$ be the cluster algebra 
 coming from  a triangulation of the surface $S_{g,n}$   [Williams 2014]  \cite[Section 3.3]{Wil1}. 
 We shall denote by  $\mathbb{A}(S_{g,n})$  the corresponding cluster $C^*$-algebra. 
 Let $T_{g,n}$ be the Teichm\"uller space of the surface $S_{g,n}$,
i.e. the set of all complex structures on $S_{g,n}$ endowed with the 
natural topology. The geodesic flow $T^t: T_{g,n}\to T_{g,n}$
is a one-parameter  group of matrices $diag ~\{ e^t, e^{-t}\}$
acting on the holomorphic quadratic differentials on the Riemann surface $S_{g,n}$. 
Such a flow gives rise to a one parameter group of automorphisms 
$\sigma_t: \mathbb{A}(S_{g,n})\to \mathbb{A}(S_{g,n})$
called the Tomita-Takesaki flow on the AF-algebra $\mathbb{A}(S_{g,n})$. 
Denote by $Prim~\mathbb{A}(S_{g,n})$ the space of all primitive ideals 
of $\mathbb{A}(S_{g,n})$ endowed with the Jacobson topology. 
Recall (\cite{Nik2}) that each primitive ideal has a parametrization by a vector 
$\Theta\in \mathbf{R}^{6g-7+2n}$ and we write it 
$I_{\Theta}\in Prim~\mathbb{A}(S_{g,n})$
\begin{theorem}\label{thm2.4}
{\bf (\cite{Nik2})}
There exists a homeomorphism
$h:  Prim~\mathbb{A}(S_{g,n})\times \mathbf{R}\to \{U\subseteq  T_{g,n} ~|~U~\hbox{{\sf is generic}}\},$
where $h(I_{\Theta},t)=S_{g,n}$ is 
given by the formula $\sigma_t(I_{\Theta})\mapsto S_{g,n}$;  the set $U=T_{g,n}$ if and only if
$g=n=1$.   The $\sigma_t(I_{\Theta})$
is an ideal of  $\mathbb{A}(S_{g,n})$ for all $t\in \mathbf{R}$ and 
 the quotient  algebra $AF$-algebra  $\mathbb{A}(S_{g,n})/\sigma_t(I_{\Theta}):=\mathbb{A}_{\Theta}^{6g-6+2n}$
is  a non-commutative coordinate ring  of  the Riemann surface  $S_{g,n}$.  
\end{theorem}

\section{Proofs}
\subsection{Proof of theorem \ref{thm1.3} }
For the sake of clarity, let us outline the main ideas. 
We start with a preparatory Lemma \ref{lm3.2} describing matrix 
 matrix $A_{\infty}\in GL_{\infty}(\mathbf{Z})$
in terms of the Jacobi-Perron continued fraction of vector $\Theta\in\mathbf{R}^{6g-7+2n}$.
Next in Lemma \ref{lm3.4} it is proved
that the abelian group $K_0(\mathcal{O}_{A_{\infty}})$
is infinite torsion, if and only if, the matrix $A_{\infty}\in GL_{\infty}(\mathbf{Z})$ is block-periodic. 
 We show in Lemma \ref{lm3.5}
 that the formulas 
$\mathbb{A}_{\Theta}\cong\mathbb{A}(S_{g,n}) / I_{\Theta}$
and  $\mathcal{O}_{A_{\infty}}\otimes\mathcal{K}\cong I_{\Theta}\rtimes_{\hat\alpha}\mathbb{T}$
imply an injective map $i:  K_0(\mathcal{O}_{A_{\infty}})\to K_0(\mathbb{A}_{\Theta})$.  
Finally, it is proved in Lemma \ref{lm3.7} that the  map $i$ 
coincides with the $m$-dimensional Minkowski question-mark function, where  $m=6g-7+2n$.  
Let us pass to a detailed argument.

We start with a preparatory lemma having an independent interest.  
Let $\Theta:=(\theta_1,\dots,\theta_m)\in \mathbf{R}^m/\mathbf{Z}^m$ and consider
the  corresponding Jacobi-Perron continued fraction [Bernstein 1971]  \cite{B}:
\begin{equation}\label{eq3.1}
\left(
\begin{matrix}
1\cr \Theta
\end{matrix}
\right)=
\lim_{k\to\infty} \left(
\begin{matrix} 
0 & 1\cr I & a_1
\end{matrix}
\right)\dots
\left(
\begin{matrix} 
0 & 1\cr I & a_k
\end{matrix}
\right)
\left(
\begin{matrix} 
0\cr \mathbb{I}
\end{matrix}
\right),
\end{equation}
where $a_i=(a^{(i)}_1,\dots, a^{(i)}_{n-1})^T$ is a vector of the non-negative integers,  
$I$ the unit matrix and $\mathbb{I}=(0,\dots, 0, 1)^T$.
Consider an infinite-dimensional matrix  $B_{\Theta}:=diag ~\{B_1, B_2, \dots\}$, where 
\begin{equation}\label{eq3.2}
B_i=
\begin{cases}
(\underbrace{1,\dots,1}_s, \underbrace{\{0,1\},\dots, \{0,1\}}_s)^T, ~\hbox{if} ~|B_i|=2s ,\cr
(\underbrace{1,\dots,1}_{s+1}, \underbrace{\{0,1\},\dots, \{0,1\}}_s)^T, ~\hbox{if} ~|B_i|=2s+1.
\end{cases}
\end{equation}
Following the pattern of formula (\ref{eq1.1}), 
the $k$-th entry $\{0,1\}$ of the blocks 
\linebreak
$B_j, \dots, B_{j+a_k^{(i)}}$
takes value $0$ and alternates  to the value $1$ for the blocks 
\linebreak
$B_{j+a_k^{(i)}+1}, \dots, B_{j+a_k^{(i)}+a_k^{(i+1)}}$. 
\begin{example}\label{ex3.1}
Assume that $S_{g,n}$ is a Riemann surface of genus $g=1$ with  $n=1$  cusps.  
The length of block $B_i$ in formula (\ref{eq3.2}) is equal to $|B_i|=2g+n=3$ in this 
case. Since the length is an odd number with $s=1$, one gets  in view of (\ref{eq3.2})  
the block $B_i=(1,1, \{0,1\})^T$.  On the other hand, since $m=6g-7+2n=1$, the Jacobi-Perron fraction in formula (\ref{eq3.1})
becomes a regular continued fraction  $[0; a_1, a_2,\dots]$ of a real number $\theta :=\theta_1$. By way of example, 
let us consider the  case $\theta=\frac{3-\sqrt{5}}{2}$.  It is well known,  that the continued fraction in this case 
must be infinite and  periodic of the form   $\theta=[0, \overline{2}]$, where $\overline{2}$ is the minimal period. 
In view of (\ref{eq1.1}),  one gets  an infinite matrix $B_{\Theta}$  given by the formula:
\begin{equation}
B_{\Theta}= diag
~\left\{ 
\left(
\begin{matrix}
1\cr 1\cr 0
\end{matrix}
\right)
\overline{
\left(
\begin{matrix}
1\cr 1\cr 1
\end{matrix}
\right)
\left(
\begin{matrix}
1\cr 1\cr 1
\end{matrix}
\right)
\left(
\begin{matrix}
1\cr 1\cr 0
\end{matrix}
\right)
\left(
\begin{matrix}
1\cr 1\cr 0
\end{matrix}
\right)
}
\right\},
\end{equation}
where bar indicates the minimal period of $B_{\Theta}$. 
\end{example}
\begin{lemma}\label{lm3.2}
 $\mathcal{O}_{B_{\Theta}}\otimes\mathcal{K}\cong I_{\Theta}\rtimes_{\hat\alpha}\mathbb{T}$,
 i.e.  $\mathcal{O}_{B_{\Theta}}$ is a Cuntz-Pimsner algebra, such that $I_{\Theta}$ is the AF-core 
 of  $\mathcal{O}_{B_{\Theta}}$. 
 \end{lemma}
\begin{proof}
(i) For the sake of clarity, consider the  case $g=n=1$. 
One gets  $m=6g-7+2n=1$ and $\Theta=(1,\theta)$. 
It is known that the cluster $C^*$-algebra $\mathbb{A}(S_{1,1})$
is given by the Bratteli diagram corresponding to the incidence matrix:
\begin{equation}\label{eq3.3}
diag
\left\{\left(
\begin{matrix}
1\cr 1\cr 1
\end{matrix}
\right)
\left(
\begin{matrix}
1\cr 1\cr 1
\end{matrix}
\right)
\dots
\right\},
\end{equation}
see [Mundici 1988] \cite[Fig. 1]{Mun1}, [Boca 2008] \cite[Fig. 2]{Boc1} and \cite[Fig. 3]{Nik2}. 

\bigskip
(ii) Let $I_{\theta}\subset \mathbb{A}(S_{1,1})$ be a primitive two-sided ideal.  
Such an ideal is given by a subgraph of the Bratteli diagram of   $\mathbb{A}(S_{1,1})$;
we refer the reader to  [Boca 2008] \cite[Fig. 7]{Boc1} for the corresponding picture. 
Clearly,  such a subgraph is obtained by cancellation of certain edges of the  Bratteli diagram for   $\mathbb{A}(S_{1,1})$.
In other words,  replacing $1$'s  by $0$'s  in the matrix (\ref{eq3.3}) gives us a matrix of incidences for the
Bratteli diagram of the ideal $I_{\theta}$. 

\bigskip
(iii) To determine modifiable  entries of (\ref{eq3.3})  
\footnote{Such entries are denoted by symbol $\{0,1\}$ in the formula (\ref{eq3.2}).}
, let $[0; a_1,a_2,\dots]$ be the regular continued fraction of $\theta\in [0,1]$. 
It follows from [Boca 2008] \cite[Fig. 7]{Boc1} that  the pattern of 
cancellation for
the edges of Bratteli diagram of $\mathbb{A}(S_{1,1})$ coincides with 
the pattern (\ref{eq1.1}) of the Minkowski question-mark function, see also 
 [Boca 2008] \cite[Remark 1]{Boc1}. 
 Indeed, the omitted edges consist of the elementary blocks $L_a$ and $R_a$ 
 shown  in  [Boca 2008] \cite[Fig. 5]{Boc1}. 
 The continued fraction  of $\theta$ gives rise to the sequence of blocks
  [Boca 2008] \cite[p. 980 at the bottom]{Boc1}:
\begin{equation}\label{eq3.4}
L_{a_1-1}\circ R_{a_2}\circ L_{a_3}\circ R_{a_2}\circ\dots
\end{equation}
It remains to compare (\ref{eq3.4}) and  the RHS of  (\ref{eq1.1})  to conclude that 
$B_{\theta}$  is given by formulas (\ref{eq3.2}), where  $s=1$ and 
$|B_i|=3$.

\bigskip
(iv)  It is verified directly  that the matrix  $B_{\theta}$ is row-finite and irreducible.
 Thus there exists a  Cuntz-Pimsner algebra $\mathcal{O}_{B_{\Theta}}$ 
satisfying the isomorphism  $\mathcal{O}_{B_{\theta}}\otimes\mathcal{K}\cong I_{\theta}\rtimes_{\hat\alpha}\mathbb{T}$. 
 This argument finishes  proof of the case $g=n=1$.  

\bigskip
(v) The general case of surface $S_{g,n}$ is treated likewise by applying  the Jacobi-Perron continued 
fractions of vector $\Theta:=(1, \theta_1,\dots,\theta_{6g-7+2n})$ [Bernstein 1971]  \cite{B}. 
The verification of formulas (\ref{eq3.2}) is direct and  left to the reader.  

\bigskip
Lemma \ref{lm3.2} is proved. 
\end{proof}

\begin{definition}\label{dfn3.3}
The matrix $A_{\infty}=diag ~\{B_1, B_2, \dots\}$ is called block-periodic,  if $B_i=B=Const$.  
\end{definition}
\begin{lemma}\label{lm3.4}
 $K_0(\mathcal{O}_{A_{\infty}})$
is an infinite torsion  group, whenever  $A_{\infty}\in GL_{\infty}(\mathbf{Z})$ is a block-periodic matrix. 
\end{lemma} 
\begin{proof}
(i)  Using Theorem \ref{thm2.3},  one gets:
\begin{equation}\label{eq3.5}
K_0(\mathcal{O}_{A_{\infty}})\cong  \frac{\mathbf{Z}^{\infty}}{(1-A_{\infty}^t) \mathbf{Z}^{\infty}}\cong
  \frac{\mathbf{Z}^{\infty}}{(1-  diag ~\{B_1^t, B_2^t, \dots\}) \mathbf{Z}^{\infty}}. 
\end{equation}

\bigskip
(ii)  Whenever $A_{\infty}$ is a block-periodic matrix, we have
$B_i=B$, where  $B=Const$ is the minimal period of  $A_{\infty}$. 
The  AF-core $\mathscr{F}
\cong\varinjlim\cup_{k,j} \cup_{i\in V_k^j} \mathscr{F}_k^j(i)\subset \mathcal{O}_{A_{\infty}}$
[Pask \& Raeburn 1996] \cite[Definition 2.2.1]{PasRae1} gives rise to the inductive limit:
\begin{equation}\label{eq3.6}
K_0(\mathcal{O}_{A_{\infty}})\cong \varinjlim
 \frac{\mathbf{Z}^{k}}{(1-  diag ~\{B^t,\dots, B^t\}) \mathbf{Z}^{k}},
\end{equation}
where $k$ is the rank of matrix $diag ~\{B^t,\dots, B^t\}$. 
Since $\det ~(1-  diag ~\{B^t,\dots, B^t\})\ne 0$, 
one concludes that  the inductive limit 
(\ref{eq3.6})  is isomorphic to an infinite torsion  group.

Lemma \ref{lm3.4} is proved. 
\end{proof}

\begin{lemma}\label{lm3.5}
There exists  an injective map 
$i:  K_0(\mathcal{O}_{A_{\infty}})\to K_0(\mathbb{A}_{\Theta})$
between the sets
\footnote{See  Remark \ref{rmk1.2}.} 
$K_0(\mathbb{A}_{\Theta})$ and $K_0(\mathcal{O}_{A_{\infty}})$.
\end{lemma} 
\begin{proof}
(i) 
Consider the Pimsner-Voiculescu exact sequence for the crossed product $C^*$-algebra   
 $I_{\Theta}\rtimes_{\hat\alpha}\mathbb{T}$ [Blackadar 1986] \cite[Section 10.6]{Bla}.
 Adopting the notation of   [Pask \& Raeburn 1996] \cite[Figure 1]{PasRae1},
 one gets the following exact sequence of the $K_0$-groups: 
 \begin{equation}\label{eq3.7}
 K_0(I_{\Theta})
 \buildrel  1-\hat\alpha_* \over\longrightarrow   K_0(I_{\Theta})
\buildrel p\over \longrightarrow K_0(\mathcal{O}_{A_{\infty}}). 
 \end{equation}
 \begin{remark}
 In [Pask \& Raeburn 1996] \cite[Figure 1]{PasRae1} the notation 
 $K_0(\mathcal{O}_{A_{\infty}}\rtimes_{\alpha}\mathbb{T})$ is used. 
 The latter is proved isomorphic to $K_0(\mathcal{O}_{A_{\infty}^{\alpha}})$
  [Pask \& Raeburn 1996] \cite[p. 432]{PasRae1}, where  
  $\mathcal{O}_{A_{\infty}^{\alpha}}$ is isomorphic to the AF-core $\mathscr{F}_{A_{\infty}}:=I_{\Theta}$ 
  of the Cuntz-Pimsner algebra $\mathcal{O}_{A_{\infty}}$ [Pask \& Raeburn 1996] \cite[Lemma 2.2.3]{PasRae1}. 
 \end{remark}

\bigskip
(ii)  On the other hand, each two-sided ideal $I_{\Theta}\subset \mathbb{A}(S_{g,n})$ 
gives rise to a short exact sequence of the $K_0$-groups 
  [Blackadar 1986] \cite[Section 2.3]{Bla}: 
 \begin{equation}\label{eq3.8}
 K_0(I_{\Theta})
 \buildrel  i \over\longrightarrow   K_0( \mathbb{A}(S_{g,n}))
\buildrel \pi\over \longrightarrow K_0(\mathbb{A}_{\Theta}). 
 \end{equation}

\begin{figure}
\begin{picture}(100,100)(0,150)

\put(0,200){\vector(1,1){15}}
\put(0,190){\vector(1,-1){15}}

\put(85,218){\vector(1,0){15}}
\put(85,172){\vector(1,0){15}}

\put(47,205){\vector(0,-1){15}}
\put(125,205){\vector(0,-1){15}}


\put(-40,193){$K_0(I_{\Theta})$}
\put(30,215){$K_0(I_{\Theta})$}
\put(25,170){$K_0( \mathbb{A}(S_{g,n}))$}

\put(110,215){$ K_0(\mathcal{O}_{A_{\infty}})$}
\put(115,170){$K_0(\mathbb{A}_{\Theta})$}

\put(-25,215){$1-\hat\alpha_*$}
\put(-10,175){$i$}
\put(30,195){$i$}
\put(135,195){$i$}

\put(90,225){$p$}
\put(90,180){$\pi$}


\end{picture}
\caption{} 
\end{figure}

\bigskip
(iii) Formulas (\ref{eq3.7}) and (\ref{eq3.8}) define a commutative diagram in Figure 1.    
The arrow closure of the diagram defines an injective map 
$i:  K_0(\mathcal{O}_{A_{\infty}})\to K_0(\mathbb{A}_{\Theta})$.

\bigskip
Lemma \ref{lm3.5} is proved. 
\end{proof}

\begin{lemma}\label{lm3.7}
The map $i:  K_0(\mathcal{O}_{A_{\infty}})\to K_0(\mathbb{A}_{\Theta})$
coincides pointwise with the   Minkowski question-mark function. 
\end{lemma} 
\begin{proof}
The proof consists in a step-by-step verification of properties (i)-(iii) of  Theorem \ref{thm2.1}
which the map (\ref{eq1.4}) must satisfy.

\bigskip
(i) Let $\left(\frac{p_1}{q_1},\dots, \frac{p_m}{q_m}\right)\in \mathbf{Q}^m/\mathbf{Z}^m$ 
be an $m$-tuple of the rational numbers. 
In view of Lemma \ref{lm3.2}, this case corresponds to a matrix $A_{\infty}= diag ~\{B\}$, where $B$ is a
finite block. In other words, $\mathcal{O}_{A_{\infty}}\cong \mathcal{O}_n$ is a Cuntz algebra. 
In particular, the isomorphism (\ref{eq2.3}) degenerates to an isomorphism:
\begin{equation}\label{eq3.9}
(A_{n^{\infty}}\otimes\mathcal{K})\rtimes_{\alpha}\mathbf{Z} \cong \mathcal{O}_n\otimes\mathcal{K}, 
\end{equation}
where $A_{n^{\infty}}$ is a uniformly hyperfinite (UHF) algebra corresponding 
to the supernatural number $n^{\infty}$ and  $\alpha$ is multiplication by $n$ automorphism 
of  $A_{n^{\infty}}$ [Blackadar  1986] \cite[Example 10.11.8]{Bla}. 
On the other hand, an analog of Lemma \ref{lm3.2} implies 
\begin{equation}\label{eq3.10}
\frac{\mathbf{Z}}{(n-1)\mathbf{Z}}\cong K_0(\mathcal{O}_n)\subset K_0(\mathbb{A}_{\Theta})\cong\mathbf{Z}^m . 
\end{equation}
Clearly, the inclusion (\ref{eq3.10}) is true if and only if $n=2$. 
Recall that $K_0(A_{2^{\infty}})\cong \mathbf{Z}[\frac{1}{2}]$ 
is the additive group of dyadic rationals
 [Blackadar  1986] \cite[Section 7.5]{Bla}.
Thus one gets a map
\begin{equation}\label{eq3.11}
\Theta=\left(\frac{p_1}{q_1},\dots, \frac{p_m}{q_m}\right)\mapsto A_{2^{\infty}} \mapsto
K_0(A_{2^{\infty}})\cong \mathbf{Z}\left[\frac{1}{2}\right].
\end{equation}
In other words, the map  (\ref{eq3.11}) implies property (i) of  Theorem \ref{thm2.1}.

\bigskip
(ii)  Let $\left(\theta_1,\dots, \theta_m\right)\in \overline{\mathbf{Q}}^m/\mathbf{Z}^m$ 
be an $m$-tuple of the algebraic numbers of degree $m+1$ over $\mathbf{Q}$. 
In this case the Jacobi-Perron fraction (\ref{eq3.1}) is periodic except, possibly, a finite
number of terms. In particular, the matrix $A_{\infty}\in GL_{\infty}(\mathbf{Z})$ 
must be block-periodic (Lemma \ref{lm3.2}). 
Therefore the group $K_0(\mathcal{O}_{A_{\infty}})$ is an infinite torsion
group (Lemma \ref{lm3.4}).  Finally, Lemma \ref{lm3.5} gives us a map:  
\begin{equation}\label{eq3.12}
 \overline{\mathbf{Q}}^m/\mathbf{Z}^m
 \ni \left(\theta_1,\dots, \theta_m\right)
\mapsto 
\left(\frac{p_1}{q_1},\dots, \frac{p_m}{q_m}\right)\in \left(\mathbf{Q}-\mathbf{Z}\left[\frac{1}{2}\right]\right)^m/\mathbf{Z}^m. 
\end{equation}
Property (ii) of  Theorem \ref{thm2.1} follows from  (\ref{eq3.12}).

\bigskip
(iii)  If $\left(\theta_1,\dots, \theta_m\right)\in \mathbf{R}^m/\mathbf{Z}^m$ 
is an $m$-tuple of the remaining irrational numbers, 
then the Jacobi-Perron continued fraction (\ref{eq3.1}) is aperiodic. 
In this case  matrix $A_{\infty}\in GL_{\infty}(\mathbf{Z})$ 
is not block-periodic and, therefore,  
the group $K_0(\mathcal{O}_{A_{\infty}})$ is no longer a torsion
group.  Thus  Lemma \ref{lm3.5} defines a map:  
\begin{equation}\label{eq3.13}
 \mathbf{R}^m/\mathbf{Z}^m
 \ni \left(\theta_1,\dots, \theta_m\right)
\mapsto 
 \left(\theta_1',\dots, \theta_m'\right)\in
\left(\mathbf{R}-\mathbf{Q}\right)^m/\mathbf{Z}^m. 
\end{equation}
Property (iii) of  Theorem \ref{thm2.1} follows from  (\ref{eq3.13}). 

\bigskip
Lemma \ref{lm3.7} is proved. 
\end{proof}

\bigskip
Theorem \ref{thm1.3} follows from Lemma \ref{lm3.7}.

\subsection{Proof of corollary \ref{cor1.4}}
Corollary \ref{cor1.4} follows from Theorem \ref{thm1.3} and \cite[Theorem 1.3]{Nik1}. 
Indeed, let $F$ be a functor on the category of $n$-dimensional projective varieties $V(k)$ with values in the 
triples $(\Lambda, [I], K)$, where  $K\subset\mathbf{R}$ is a  number field, 
$\Lambda\subseteq O_K$  is an order in the ring of integers  $O_K$ of the field $K$  and $[I]\subset\Lambda$
is an ideal class  \cite[Theorem 1.3]{Nik1}.
Moreover, one gets a grading  $\{(\Lambda^i, [I], K) ~|~ 0\le i\le 2n\}$ by  the subrings $\Lambda^i\subseteq\Lambda$ as specified 
in \cite[p. 271]{Nik3}.
Likewise,   Theorem \ref{thm1.3}  defines  a functor on the
triples  $(\Lambda^i, [I], K)\cong  (K_0(\mathbb{A}_{\Theta_i}), K_0^+(\mathbb{A}_{\Theta_i}))$ \cite[Section 7.3]{Bla}
with values in the similarity classes of matrices $\{A^i_{\infty}\in GL_{\infty}(\mathbf{Z})  ~|~0\le i\le 2n\}$.  
On the other hand,  Serre and Deninger used an infinite-dimensional cohomology $H_{ar}^i(V)$ 
to define a functor on $V(k)$ with values in the similarity classes of matrices   
$\{Fr^i_{\infty}\in GL_{\infty}(\mathbf{Z})  ~|~0\le i\le 2n\}$ . The latter represent Frobenius 
action on  the cohomology groups $H_{ar}^i(V)$  [Deninger 1991]  \cite[Theorem 4.1]{Den1} and
[Serre 1970] \cite{Ser1}. Thus  one gets a commutative diagram in Figure 2. 
It follows from the diagram that
$\{Fr_{\infty}^i\cong A_{\infty}^i ~|~ 0\le i\le 2n\}$,
where $\cong$ is the similarity of matrices in the group $GL_{\infty}(\mathbf{Z})$.

\bigskip
Corollary \ref{cor1.4} is proved. 

\begin{figure}
\begin{picture}(100,100)(60,150)

\put(85,218){\vector(1,0){45}}
\put(85,172){\vector(1,0){45}}

\put(47,205){\vector(0,-1){20}}
\put(185,205){\vector(0,-1){20}}

\put(35,215){$V(k)$}
\put(25,170){$(\Lambda, [I], K)$}

\put(145,215){$\{Fr^i_{\infty} ~|~0\le i\le 2n\}$}
\put(145,170){$\{A^i_{\infty} ~|~0\le i\le 2n\} $}

\put(30,195){$F$}

\put(85,223){{\tiny Serre-Deninger}}
\put(88,207){\tiny {cohomology}}
\put(87,178){{\tiny Theorem \ref{thm1.3}}}

\put(190,190){{\tiny matrix similarity}}
\put(205,197){$\cong$}


\end{picture}
\caption{} 
\end{figure}

  
   
  

\bibliographystyle{amsplain}


\end{document}